\documentclass[a4paper]{amsart}

\usepackage[utf8]{inputenc}
\usepackage[T1]{fontenc}
\usepackage{amsthm}
\usepackage{xspace}
\usepackage{amssymb,amsmath,stmaryrd}
\usepackage[english]{babel}
\usepackage[shortlabels]{enumitem}
\usepackage[all]{xy}
\usepackage{xcolor}
\usepackage{hyperref}

\theoremstyle{plain}
\newtheorem{theorem}{Theorem}[section]
\newtheorem{remark}[theorem]{Remark}
\newtheorem{lemma}[theorem]{Lemma}

\newtheorem{proposition}[theorem]{Proposition}

\theoremstyle{definition}

\newtheorem{definition}[theorem]{Definition}
\newtheorem{notation}[theorem]{Notation}
\newtheorem{example}[theorem]{Example}
\numberwithin{equation}{section}

\newcommand{\ca}{\mbox{$C\sp*$-}al\-ge\-bra\xspace}
\newcommand{\cas}{\mbox{$C\sp*$-}al\-ge\-bras\xspace}
\newcommand{\starhomo}{\mbox{$\sp*$-}ho\-mo\-morphism\xspace}

\newcommand{\stariso}{\mbox{$\sp*$-}iso\-morphism\xspace}

\newcommand{\Z}{\ensuremath{\mathbb{Z}}\xspace}
\newcommand{\N}{\ensuremath{\mathbb{N}}\xspace}

\newcommand{\K}{\ensuremath{\mathbb{K}}\xspace}
\newcommand{\ie}{\emph{i.e.}\xspace}
\newcommand{\cf}{\emph{cf.}\xspace}

\newcommand{\kk}{\ensuremath{\operatorname{KK}}\xspace}

\newcommand{\A}{\ensuremath{\mathfrak{A}}\xspace}

\newcommand{\Asf}{\mathsf{A}}
\newcommand{\Bsf}{\mathsf{B}}
\newcommand{\Esf}{\mathsf{E}}

\newcommand{\setof}[2]{\left\{ #1 \;\middle|\; #2 \right\}}

\newcommand{\Meq}{\ensuremath{\sim_{M\negthinspace E}}\xspace}

\newcommand{\OO}{\mbox{\texttt{\textup{(O)}}}\xspace}
\newcommand{\CO}{\mbox{\texttt{\textup{(Col)}}}\xspace}
\newcommand{\II}{\mbox{\texttt{\textup{(I)}}}\xspace}
\newcommand{\RR}{\mbox{\texttt{\textup{(R)}}}\xspace}
\newcommand{\SSS}{\mbox{\texttt{\textup{(S)}}}\xspace}
\newcommand{\CC}{\mbox{\texttt{\textup{(C)}}}\xspace}

\title{Invariance of the Cuntz splice}
\date{\today}
\author{S\o{}ren Eilers}
\address{Department of Mathe\-matical Sciences, University of Copen\-hagen, Universi\-tets\-park\-en~5, DK-2100 Copen\-hagen, Den\-mark}
\email{eilers@math.ku.dk}
\author{Gunnar Restorff}
\address{Department of Science and Technology, University of the Faroe Islands, N\'{o}at\'{u}n~3, FO-100 T\'{o}rshavn, the Faroe Islands}
\email{gunnarr@setur.fo}
\author{Efren Ruiz}
\address{Department of Mathematics, University of Hawaii, Hilo, 200 W.~Kawili St., Hilo, Hawaii, 96720-4091 USA}
\email{ruize@hawaii.edu}
\author{Adam P.~W.~S\o{}rensen}
\address{Department of Mathematics, University of Oslo, PO BOX 1053 Blindern, N-0316 Oslo, Norway}
\email{apws@math.uio.no}

\begin{document}


\begin{abstract}
We show that the Cuntz splice induces stably isomorphic graph $C^*$-algebras.
\end{abstract}

\maketitle

\section{Introduction}

Cuntz and Krieger introduced the Cuntz-Krieger algebras in \cite{MR561974}, and Cuntz showed in \cite{MR608527} that if we restrict to the matrices satisfying the modest condition (II), then the stabilized Cuntz-Krieger algebras are an invariant of shifts of finite type up to flow equivalence. Shortly after Franks had made a successful classification of irreducible shifts of finite type up to flow equivalence (\cite{MR758893}), Cuntz raised the question of whether this invariant or the $K_0$-group alone classifies simple Cuntz-Krieger algebras up to stable isomorphism. He sketched in \cite{MR866492} that it was enough to answer whether $\mathcal{O}_2$ and $\mathcal{O}_{2_-}$ are isomorphic, where $\mathcal{O}_2$ and $\mathcal{O}_{2_-}$ are the Cuntz-Krieger algebras associated with the matrices
$$\begin{pmatrix} 1 & 1 \\ 1 & 1 \end{pmatrix}\quad\text{and}\quad\begin{pmatrix} 1 & 1 & 0 & 0 \\ 1 & 1 & 1 & 0 \\ 0 & 1 & 1 & 1 \\ 0 & 0 & 1 & 1 \end{pmatrix},$$
respectively. This question remained open until R\o{}rdam in \cite{MR1340839} showed that $\mathcal{O}_2$ and $\mathcal{O}_{2_-}$ are in fact isomorphic and elaborated on the arguments of Cuntz to show that the $K_0$-group is a complete invariant of the stabilized simple Cuntz-Krieger algebras. 

This procedure of gluing the graph corresponding to the former matrix above onto another graph has since been known as \emph{Cuntz splicing} a graph at a certain vertex. Knowing that when we Cuntz splice a graph (on a vertex that supports two return paths), we get stably isomorphic Cuntz-Krieger algebras, has been important for classifying Cuntz-Krieger algebras (\cite{MR1340839,MR1329907,MR2270572}), as well as understanding the connection between the dynamics of the underlying shift spaces and the Cuntz-Krieger algebras.  With the recent work on the relation between move equivalence of graphs and stable isomorphism of the corresponding graph \cas, the question of whether Cuntz splicing yields stably isomorphic \cas has become of great interest. Bentmann has shown that this is in fact the case for purely infinite graph \cas with finitely many ideals (\cite{arXiv:1510.06757v2}), while Gabe recently has generalized this to also cover general purely infinite graph \cas (\cite{GabePrivate2016}). Their methods depend heavily on the result of Kirchberg on lifting invertible ideal-related \kk-elements to equivariant isomorphisms for strongly purely infinite \cas (\cite{MR1796912}).

In this paper we show in general that Cuntz splicing a vertex that supports two distinct return paths yields stably isomorphic graph \cas{} --- only assuming that the graph is countable. The results, the proofs and methods of this paper are important for recent development in the geometric classification of general Cuntz-Krieger algebras and of unital graph \cas (\cite{Eilers-Restorff-Ruiz-Sorensen-1}, \cite{Eilers-Restorff-Ruiz-Sorensen-2}) as well as for the question of strong classification of general Cuntz-Krieger algebras and of unital graph \cas (\cite{Carlsen-Restorff-Ruiz}, \cite{Eilers-Restorff-Ruiz-Sorensen-2}).  

We proved invariance of the Cuntz splice in the special case of unital graph \cas in an arXiv preprint (1505.06773) posted in May 2015. Bentmann's recent paper showed us how to reduce the general question to the row-finite case, and we proceeded to discover that our arguments applied with only minor changes to that case. Since most of the results of our preprint have since been superseded by other forthcoming work, we do not intend to publish it, whereas this work is intended for publication.

\section{Preliminaries}

\begin{definition}
A graph $E$ is a quadruple $E = (E^0 , E^1 , r, s)$ where $E^0$ and $E^1$ are sets, and $r$ and $s$ are maps from $E^1$ to $E^0$. 
The elements of $E^0$ are called \emph{vertices}, the elements of $E^1$ are called \emph{edges}, the map $r$ is called the \emph{range map}, and the map $s$ is called the \emph{source map}. 
\end{definition}

When working with several graphs at the same time, to avoid confusion, we will denote the range map and source map of a graph $E$ by $r_E$ and $s_E$ respectively.

All graphs considered will be \emph{countable}, \ie, there are countably many vertices and edges. 

\begin{definition}
A \emph{loop} is an edge with the same range and source. 

A \emph{path} $\mu$ in a graph is a finite sequence $\mu = e_1 e_2 \cdots e_n$ of edges satisfying 
$r(e_i)=s(e_{i+1})$, for all $i=1,2,\ldots, n-1$, and we say that the \emph{length} of $\mu$ is $n$. 
We extend the range and source maps to paths by letting $s(\mu) = s(e_1)$ and $r(\mu) = r(e_n)$. 
Vertices in $E$ are regarded as \emph{paths of length $0$} (also called empty paths). 

A \emph{cycle} is a nonempty path $\mu$ such that $s(\mu) = r(\mu)$.
We call a cycle $e_1e_2\cdots e_n$ a \emph{vertex-simple cycle} if $r(e_i)\neq r(e_j)$ for all $i\neq j$. A vertex-simple cycle $e_1e_2\cdots e_n$ is said to have an \emph{exit} if there exists an edge $f$ such that $s(f)=s(e_k)$ for some $k=1,2,\ldots,n$ with $e_k\neq f$. 
A \emph{return path} is a cycle $\mu = e_1 e_2 \cdots e_n$ such that $r(e_i) \neq r(\mu)$ for $i < n$.

For a loop, cycle or return path, we say that it is \emph{based} at the source vertex of its path. 
We also say that a vertex \emph{supports} a certain loop, cycle or return path if it is based at that vertex. 

Note that in \cite{MR1988256,MR1914564}, the authors use the term \emph{loop} where we use \emph{cycle}.
\end{definition}

\begin{definition}
A vertex $v$ in $E$ is called \emph{regular} if $s^{-1}(v)$ is finite and nonempty.  We denote the set of regular vertices by $E_{\mathrm{reg}}^0$.

A vertex $v$ in $E$ is called a \emph{sink} if $s^{-1}(v)=\emptyset$.  A graph $E$ is called \emph{row-finite} if for each $v \in E^0$, $v$ is either a sink or a regular vertex.
\end{definition}

It is essential for our approach to graph \cas to be able to shift between a graph and its adjacency matrix. 
In what follows, we let \N denote the set of positive integers, while $\N_0$ denotes the set of nonnegative integers.

\begin{definition}
Let $E = (E^0 , E^1 , r, s)$ be a graph.
We define its \emph{adjacency matrix} $\Asf_E$ as a $E^0\times E^0$ matrix with the $(u,v)$'th entry being
$$\left\vert\setof{e\in E^1}{s(e)=u, r(e)=v}\right\vert.$$
As we only consider countable graphs, $\Asf_E$ will be a finite matrix or a countably infinite matrix, and it will have entries from $\N_0\sqcup\{\infty\}$.

Let $X$ be a set.
If $A$ is an $X \times X$ matrix with entries from $\N_0\sqcup\{\infty\}$, we let $\Esf_{A}$ be the graph with vertex set $X$ such that between two vertices $x,x' \in X$ we have $A(x,x')$ edges.
\end{definition}

It will be convenient for us to alter the adjacency matrix of a graph in a very specific way, subtracting the identity, so we introduce notation for this. 

\begin{notation}
Let $E$ be a graph and $\Asf_E$ its adjacency matrix.  Let $\Bsf_E$ denote the matrix $\Asf_{E} - I$.
\end{notation}

\subsection{Graph \texorpdfstring{$C^*$}{C*}-algebras}
We follow the notation and definition for graph \cas in \cite{MR1670363}; this is not the convention used in Raeburn's monograph \cite{MR2135030}. 

\begin{definition} \label{def:graphca}
Let $E = (E^0,E^1,r,s)$ be a graph.
The \emph{graph \ca} $C^*(E)$ is defined as the universal \ca generated by
a set of mutually orthogonal projections $\setof{ p_v }{ v \in E^0 }$ and a set $\setof{ s_e }{ e \in E^1 }$ of partial isometries satisfying the relations
\begin{itemize}
	\item $s_e^* s_f = 0$ if $e,f \in E^1$ and $e \neq f$,
	\item $s_e^* s_e = p_{r(e)}$ for all $e \in E^1$,
	\item $s_e s_e^* \leq p_{s(e)}$ for all $e \in E^1$, and,
	\item $p_v = \sum_{e \in s^{-1}(v)} s_e s_e^*$ for all $v \in E^0$ with $0 < |s^{-1}(v)| < \infty$.
\end{itemize}
Whenever we have a set of mutually orthogonal projections $\setof{ p_v }{ v \in E^0 }$ and a set $\setof{ s_e }{ e \in E^1 }$ of partial isometries in a \ca satisfying the relations, then we call these elements a \emph{Cuntz-Krieger $E$-family}. 
\end{definition}

We will also need moves on graphs as defined in \cite{MR3082546}.  In the case of graphs with finitely many vertices the basic moves are outsplitting (Move \OO), insplitting (Move \II), reduction (Move \RR), and removal of a regular source (Move \SSS). It turns out that in the general setting, move \RR must be replaced by the following

\begin{definition}[Collapse a regular vertex that does not support a loop, Move \CO ] \label{def:collapse}
Let $E = (E^0 , E^1 , r , s )$ be a graph and let $v$ be a regular vertex in $E$ that does not support a loop. 
Define a graph $E_{COL}$ by 
\begin{align*}
E_{COL}^0 &= E^0 \setminus \{v\}, \\
E_{COL}^1 &= E^1 \setminus (r^{-1}(v) \cup s^{-1}(v))
\sqcup \setof{[ef ]}{e \in r^{-1}(v)\text{ and }f \in s^{-1}(v)},
\end{align*}
the range and source maps extends those of $E$, and satisfy $r_{E_{COL}}([ef ]) = r (f )$
and $s_{E_{COL}} ([ef ]) = s (e)$.
\end{definition}

Move  \CO was defined in \cite[Theorem~5.2]{MR3082546} for graphs with finitely many vertices as an auxiliary move, and proved there to be realized by moves \II, \OO and \RR. 

\begin{definition}
The equivalence relation generated by the moves \OO, \II, \SSS, \CO together with graph isomorphism is called \emph{move equivalence}, and denoted \Meq. 

Let $X$ be a set and let $A$ and $A'$ be $X \times X$ matrices with entries from $\N_0\sqcup\{\infty\}$.  If  $\Esf_A \Meq \Esf_{A'}$, then we say that $A$ and $A'$ are \emph{move equivalent}, and we write $A \Meq A'$.
\end{definition}

\begin{remark}
 By \cite[Theorem~5.2]{MR3082546}, the above definition is equivalent to the definition in \cite[Section~4]{MR3082546} for graphs with finitely many vertices.
\end{remark}

These moves have been considered by other authors, and were previously noted to preserve the Morita equivalence class of the associated graph \ca.  The moves \OO and \II induce stably isomorphic $C^*$-algebras due to the results in \cite{MR2054048}, and by \cite{MR2215769}, moves \RR, \SSS, \CO preserve the Morita equivalence class of the associated graph \cas (see also \cite[Propositions~3.1, 3.2 and 3.3 and Theorem~3.5]{MR3082546}).  Therefore, we get the following theorem.

\begin{theorem}\label{thm:moveimpliesstableisomorphism}
Let $E_1$ and $E_2$ be graphs such that $E_1\Meq E_2$. 
Then $C\sp*(E_1)\otimes \K\cong C\sp*(E_2)\otimes \K$. 
\end{theorem}

We now recall the definiton of the Cuntz splice (see Notation \ref{notation:OnceAndTwice-new} and Example \ref{example:cuntz-splice-new} for illustrations).

\begin{definition}[Move \CC: Cuntz splicing at a regular vertex supporting two return paths] \label{def:cuntzsplice}
Let $E = (E^0 , E^1 , r , s )$ be a graph and let $v \in E^0$ be a regular vertex that supports at least two return paths.
Let ${E_{v,-}}$ denote the graph $(E_{v,-}^0 , E_{v,-}^1 , r_{v,-}, s_{v,-})$ defined by 
\begin{align*}
{E^0_{v,-}} &:= E^0\sqcup\{u_1 , u_2 \} \\
{E^1_{v,-}} &:= E^1\sqcup\{e_1 , e_2 , f_1 , f_2 , h_1 , h_2 \},
\end{align*}
where $r_{v,-}$ and $s_{v,-}$ extend $r$ and $s$, respectively, and satisfy
$$s_{v,-} (e_1 ) = v,\quad s_{v,-} (e_2 ) = u_1 ,\quad s_{v,-} (f_i ) = u_1 ,\quad s_{v,-} (h_i ) = u_2 ,$$
and
$$r_{v,-} (e_1 ) = u_1 ,\quad r_{v,-} (e_2 ) = v,\quad r_{v,-} (f_i ) = u_i ,\quad r_{v,-} (h_i ) = u_i . $$
We call ${E_{v,-}}$ the \emph{graph obtained by Cuntz splicing $E$ at $v$}, and say ${E_{v,-}}$ is formed by performing Move \CC to $E$. 
\end{definition}

The  aim of this paper is to prove that  $C\sp*(E)\otimes \K\cong C\sp*({E_{v,-}})\otimes \K$ for any graph $E$. In fact, we prove slightly more, since our proof allows for Cuntz splicing also at infinite emitters supporting at least two return paths.

\section{Elementary matrix operations preserving move equivalence}\label{addmoves}

In this section we perform row and column additions on $\Bsf_E$ without changing the move equivalence class of the associated graphs. 
Our setup is slightly different from what was considered in \cite[Section 7]{MR3082546}, so we redo the proofs from there in our setting. 
There are no substantial changes in the proof techniques, which essentially go back to \cite{MR758893}.

\begin{lemma} \label{lem:oneStepColumnAdd}
Let $E = (E^0, E^1, r_E,s_E)$ be a graph. 
Let $u,v \in E^0$ be distinct vertices. 
Suppose the $(u,v)$'th entry of $\Bsf_E$ is nonzero (\ie, there is an edge from $u$ to $v$), and that the sum of the entries in the $u$'th row of $\Bsf_E$ is strictly greater than 0 (\ie, $u$ emits at least two edges). 
If $B'$ is the matrix formed from $\Bsf_E$ by adding the $u$'th column into the $v$'th column, then 
$$\Asf_E \Meq B' + I. $$
\end{lemma}
\begin{proof}
Fix an edge $f$ from $u$ to $v$. 
Form a graph $G$ from $E$ by removing $f$ but adding for each edge $e \in r_E^{-1}(u)$ an edge $\bar{e}$ with $s_G(\bar{e}) = s_E(e)$ and $r_G(\bar{e}) = v$. 
We claim that $B' = \Bsf_G$. 
At any entry other than the $(u,v)$'th entry the two matrices have the same values, since we in both cases add entries into the $v$'th column that are exactly equal to the number of edges in $E$. 
At the $(u,v)$'th entry of $\Bsf_G$ we have
\begin{align*}
	(|s_E^{-1}(u) \cap r_E^{-1}(v)| - 1) + |s_E^{-1}(u) \cap r_E^{-1}(u)| &= \Bsf_E(u,v) + \Bsf_E(u,u) = B'(u,v).
\end{align*}
Thus to prove this lemma it suffices to show $E \Meq G$.

Partition $s_E^{-1}(u)$ as $\mathcal{E}_1 = \{ f \}$ and $\mathcal{E}_2 = s_E^{-1}(u) \setminus \{ f \}$. 
By assumption $\mathcal{E}_2$ is not empty, so we can use Move \OO. 
Doing so yields a graph just as $E$ but where $u$ is replaced by two vertices, $u_1$ and $u_2$.
The vertex $u_1$ receives a copy of everything $u$ did and it emits only one edge.  
That edge has range $v$. 
The vertex $u_2$ also receives a copy of everything $u$ did, and it emits everything $u$ did, except $f$. 
Since $u_1$ is regular and not the base of a loop, we can collapse it.
The resulting graph is $G$ (after we relabel $u_2$ as $u$), so $G \Meq E$.
\end{proof}

We can also add columns along a path. 

\begin{proposition} \label{prop:columnAdd}
Let $E = (E^0, E^1, r_E,s_E)$ be a graph and let $u,v \in E^0$ be distinct vertices with a path from $u$ to $v$ going through distinct vertices $u = u_0, u_1, u_2, \ldots, u_n = v$ (labelled so there is an edge from $u_i$ to $u_{i+1}$ for $i=0,1,2 \ldots,n-1$).  Suppose further that $u$ supports a loop.  If $B'$ is the matrix formed from $\Bsf_E$ by adding the $u$'th column into the $v$'th column, then 
\[
	\Asf_E \Meq B' + I. 
\]
\end{proposition}
\begin{proof}
That $u$ supports a loop guarantees that $B'+I$ is the adjacency matrix of a graph $E'=\Esf_{B'+I}$. 

The vertex $u_i$ emits exactly one edge in $E$ if and only if it emits exactly one edge in $E'$, for $i=1,\ldots,n-1$. So by collapsing all regular vertices $u_i$, $i=1,2,\ldots,n-1$ emitting exactly one edge both in $E$ and in $E'$, we get two new graphs $E_1\Meq E$ and $E_1'\Meq E'$. 
In $E_1$, there is a path from $u$ to $v$ through vertices that all emit at least two edges. 
Moreover, $\Bsf_{E_1'}$ is obtained from $\Bsf_{E_1}$ by adding the $u$'th column into the $v$'th column. 
Therefore, we may without loss of generality assume that all the vertices $u_i$, $i=0,1,2,\ldots,n-1$ emit at least two edges. 

By repeated applications of Lemma~\ref{lem:oneStepColumnAdd}, we first add the $u_{n-1}$'th column into the $u_n$'th column of $\Bsf_E$, which we can since there is an edge from $u_{n-1}$ to $u_n$.
Then we add the $u_{n-2}$'th column into the $u_n$'th column, which we can since there now is an edge from $u_{n-2}$ to $u_n$. 
Continuing this way, we end up with a matrix $C$ which is formed from $\Bsf_E$ by adding all the columns $u_i$, for $i = 0,1,2,\ldots, n-1$, into the the $u_n$'th column. 
We have that $\Asf_E \Meq C + I$.

Now consider the matrix $B'=\Bsf_{E'}$. 
By repeated applications of Lemma~\ref{lem:oneStepColumnAdd}, we first add the $u_{n-1}$'th column into the $u_n$'th column of $B'=\Bsf_{E'}$, which we can since there is an edge from $u_{n-1}$ to $u_n$.
Then we add the $u_{n-2}$'th column into the $u_n$'th column, which we can since there now is an edge from $u_{n-2}$ to $u_n$. 
Continuing this way, we end up with a matrix $D$ which is formed from $B'=\Bsf_{E'}$ by adding all the columns $u_i$, for $i = 1,2,\ldots, n-1$, into the the $u_n$'th column. 
We have that $B'+I=\Asf_{E'} \Meq D + I$.

But it is clear from the construction that $C=D$. 
\end{proof}

\begin{remark} \label{rmk:columnAdd}
Similar to how we used Lemma \ref{lem:oneStepColumnAdd} in the above proof, we can use Proposition \ref{prop:columnAdd} ``backwards'' to subtract columns in $\Bsf_E$ as long as the addition that undoes the subtraction would be legal. 
\end{remark}

We now turn to row additions. 

\begin{lemma} \label{lem:oneStepRowAdd}
Let $E = (E^0, E^1, r_E,s_E)$ be a graph. 
Let $u,v \in E^0$ be distinct vertices. 
Suppose the $(v,u)$'th entry of $\Bsf_E$ is nonzero (\ie, there is an edge from $v$ to $u$), and that $u$ is a regular vertex. 
If $B'$ is the matrix formed from $\Bsf_E$ by adding the $u$'th row into the $v$'th row, then 
\[
	\Asf_E \Meq B' + I. 
\]
\end{lemma}
\begin{proof}
Let $E'=\Esf_{B'+I}$ denote the graph with adjacency matrix $B'+I$. 

First assume that $u$ only receives one edge in $E$ (which necessarily is the edge from $v$). 
Then $u$ is a regular vertex not supporting a loop, so we can collapse it obtaining a graph $E''$. 
Note that the vertex $u$ is a regular source in $E'$, so we may remove it. It is clear that the resulting graph is exactly $E''$.  

Now assume instead that $u$ receives at least two edges. 
Fix an edge $f$ from $v$ to $u$.
Form a graph $G$ from $E$ by removing $f$ but adding for each edge $e \in s_E^{-1}(u)$ an edge $\bar{e}$ with $s_G(\bar{e}) = v$ and $r_G(\bar{e}) = r_E(e)$. 
We claim that $E \Meq G$. 
Arguing as in the proof of Lemma \ref{lem:oneStepColumnAdd} we see that this is equivalent to proving $\Asf_E \Meq B' + I$.

Partition $r_E^{-1}(u)$ as $\mathcal{E}_1 = \{ f \}$ and $\mathcal{E}_2 = r_E^{-1}(u) \setminus \{ f \}$. 
By our assumptions on $u$, $\mathcal{E}_2$ is nonempty, and $u$ is regular, so we can use Move \II.  
Doing so replaces $u$ with two new vertices, $u_1$ and $u_2$. 
The vertex $u_1$ only receives one edge, and that edge comes from $v$, the vertex $u_2$ receives the edges $u$ received except $f$.
Since $u_1$ is regular and not the base of a loop we can collapse it.
The resulting graph is $G$ (after we relabel $u_2$ as $u$), so $G \Meq E$.
\end{proof}

We can also add rows along a path of vertices.

\begin{proposition} \label{prop:rowAdd}
Let $E = (E^0, E^1, r_E,s_E)$ be a graph and let $u,v \in E^0$ be distinct vertices with a path from $v$ to $u$ going through distinct vertices $v = v_0, v_1, v_2, \ldots, v_n = u$ (labelled so there is an edge from $v_i$ to $v_{i+1}$ for $i=0,1,2,\ldots,n-1$). 
Suppose further that the vertex $u$ is regular and  supports at least one loop.
If $B'$ is the matrix formed from $\Bsf_E$ by adding the $u$'th row into the $v$'th row, then 
\[
	\Asf_E \Meq B' + I. 
\]
\end{proposition}
\begin{proof}
That $u$ supports a loop guarantees that $B'+I$ is the adjacency matrix of a graph $E'=\Esf_{B'+I}$. 

First we prove the special case where all the vertices $v_1,\ldots,v_n$ are regular. 
By repeated applications of Lemma~\ref{lem:oneStepRowAdd}, we first add the $v_{1}$'st row into the $v_0$'th row of $\Bsf_E$, which we can since there is an edge from $v_{0}$ to $v_1$ and $v_1$ is regular.
Then we add the $v_{2}$'nd row into the $v_0$'th row, which we can since there now is an edge from $v_{0}$ to $v_2$ and $v_2$ is regular. 
Continuing this way, we end up with a matrix $C$ which is formed from $\Bsf_E$ by adding all the rows $v_i$, for $i = 1,2,\ldots, n$, into the the $v_0$'th column. 
We have that $\Asf_E \Meq C + I$.

Now consider the matrix $B'=\Bsf_{E'}$. 
By repeated applications of Lemma~\ref{lem:oneStepRowAdd}, we first add the $v_{1}$'st row into the $v_0$'th row of $B'=\Bsf_{E'}$, which we can since there is an edge from $v_{0}$ to $v_1$.
Then we add the $v_{2}$'nd row into the $v_0$'th row, which we can since there now is an edge from $v_{0}$ to $v_2$. 
Continuing this way, we end up with a matrix $D$ which is formed from $B'=\Bsf_{E'}$ by adding all the rows $v_i$, for $i = 1,2,\ldots, n-1$, into the the $v_0$'th row. 
We have that $B'+I=\Asf_{E'} \Meq D + I$.
But it is clear from the construction that $C=D$.

Now we prove that the general case when only $u$ is assumed to be regular can be reduced to the case where $v_1,\ldots,v_n$ are regular. Choose a path $e_0e_1\cdots e_{n-1}$ going through the distinct vertices $v_1,\ldots,v_n$. For each singular vertex $v_{i}$, $i=1,\ldots,n-1$, we outsplit according to the partition $\mathcal{E}_i^1=\{e_i\}$ and $\mathcal{E}_i^2=s_E^{-1}(v_i)$ and call the corresponding vertices $v_i^1$ and $v_i^2$, respectively. Denote the split graph by $E_1$, and denote the vertices $v_i$, $i=1,\ldots,n-1$ that were not split by $v_i^1$. Note that we now have a path from $v$ to $u$ through distinct regular vertices. Note also that since all vertices along the path are distinct, what happens to the $v_i$'th entry of row $u$ and $v$ is that it gets doubled for each vertex $u_i$ that gets split and stays unchanged for the vertices $u_i=u_i^1\in E^0$ that are regular. Let $E'$ be the graph $\Esf_{B'+I}$, and let $E_1'$ be the graph constructed using exactly the same outsplittings as in the graph above. Now it is clear that the graph we get from $E_1$ by adding row $u$ into row $v$ is exactly $E_1'$. Thus the general case now follows from the above. 
\end{proof}

\begin{remark} \label{rmk:rowAdd}
We can also use Proposition \ref{prop:rowAdd} ``backwards'' to subtract rows in $\Bsf_E$ (\cf\ Remark \ref{rmk:columnAdd}). 
\end{remark}

\section{Cuntz splice implies stable isomorphism}
\label{sec:cuntzsplice}

In this section, we show that the Cuntz splice gives stably isomorphic graph $C^*$-algebras.  
We first set up some notation.  

\begin{notation}\label{notation:OnceAndTwice-new}
Let $\mathbf{E}_*$ and $\mathbf{E}_{**}$ denote the graphs: 
\begin{align*}
\mathbf{E}_* \  = \ \ \ \ \xymatrix{
  \bullet^{v_1} \ar@(ul,ur)[]^{e_{1}} \ar@/^/[r]^{e_{2}} & \bullet^{v_2} \ar@(ul,ur)[]^{e_{4}} \ar@/^/[l]^{e_{3}}
}
\end{align*}
\begin{align*}
\mathbf{E}_{**} \  =  \ \ \ \ \xymatrix{
	\bullet^{ w_{4} } \ar@(ul,ur)[]^{f_{10}}  \ar@/^/[r]^{ f_{9} } & \bullet^{ w_{3} } \ar@(ul,ur)[]^{f_{7}} \ar@/^/[r]^{ f_{6} }  \ar@/^/[l]^{f_{8}} &  \bullet^{w_1} 				\ar@(ul,ur)[]^{f_{1}} \ar@/^/[r]^{f_{2}} \ar@/^/[l]^{f_{5}}
	& \bullet^{w_2} \ar@(ul,ur)[]^{f_{4}} \ar@/^/[l]^{f_{3}}
	}
\end{align*}
The graph $\mathbf{E}_*$ is what we attach when we Cuntz splice.  If we instead attach the graph $\mathbf{E}_{**}$, we have Cuntz spliced twice.

Let $E = ( E^{0}, E^{1} , r_{E}, s_{E} )$ be a graph and let $u$ be a vertex of $E$.
Then $E_{u, -}$ can be described as follows (up to canonical isomorphism):
\begin{align*}
E_{u,-}^{0} &= E^{0} \sqcup \mathbf{E}_{*}^{0} \\
E_{u,-}^{1} &= E^{1} \sqcup \mathbf{E}_{*}^{1} \sqcup \{ d_1, d_2 \}
\end{align*}
with $r_{E_{u,-}} \vert_{E^{1}} = r_{E}$, $s_{E_{u,-}} \vert_{ E^{1} } = s_{E}$, $r_{E_{u,-}} \vert_{\mathbf{E}_{*}^{1}} = r_{\mathbf{E}_{*}}$, $s_{E_{u,-}} \vert_{\mathbf{E}_{*}^{1}} = s_{\mathbf{E}_{*}}$, and
\begin{align*}
	s_{E_{u,-}}(d_1) &= u	& r_{E_{u,-}}(d_1) &= v_{1} \\
	s_{E_{u,-}}(d_2) &= v_1	& r_{E_{u,-}}(d_2) &= u.
\end{align*}
Moreover, $E_{u,--}$ can be described as follows (up to canonical isomorphism):
\begin{align*}
E_{u,--}^{0} &= E^{0} \sqcup \mathbf{E}_{**}^{0} \\
E_{u,--}^{1} &= E^{1} \sqcup \mathbf{E}_{**}^{1} \sqcup \{ d_1, d_2 \}
\end{align*}
with $r_{E_{u,--}} \vert_{E^{1}} = r_{E}$, $s_{E_{u,--}} \vert_{ E^{1} } = s_{E}$, $r_{E_{u,--}} \vert_{\mathbf{E}_{**}^{1}} = r_{\mathbf{E}_{**}}$, $s_{E_{u,--}} \vert_{\mathbf{E}_{**}^{1}} = s_{\mathbf{E}_{**}}$, and
\begin{align*}
	s_{E_{u,--}}(d_1) &= u		& r_{E_{u,--}}(d_1) &= w_{1} \\
	s_{E_{u,--}}(d_2) &= w_1	& r_{E_{u,--}}(d_2) &= u.
\end{align*}
\end{notation}

\begin{example}\label{example:cuntz-splice-new}
Consider the graph 
\begin{align*}
E \  = \ \ \ \ \xymatrix{
  \bullet_{u} \ar@(dl,ul) \ar@(ur,dr)
}
\end{align*}
Then 
\begin{align*}
E_{u,-} \  = \ \ \ \ \xymatrix{
  \bullet^{v_1} \ar@/^/[d]^{d_2} \ar@(ul,ur)[]^{e_{1}} \ar@/^/[r]^{e_{2}} & \bullet^{v_2} \ar@(ul,ur)[]^{e_{4}} \ar@/^/[l]_{e_{3}} \\
  \bullet_{u} \ar@/^/[u]^{d_1} \ar@(dl,ul) \ar@(ur,dr) & 
}
\end{align*}
and 
\begin{align*}
E_{u,--} \  = \ \ \ \ \xymatrix{
	\bullet^{ w_{4} } \ar@(ul,ur)[]^{f_{10}}  \ar@/^/[r]^{ f_{9} } &
	\bullet^{ w_{3} } \ar@(ul,ur)[]^{f_{7}} \ar@/^/[r]^{ f_{6} }  \ar@/^/[l]_{f_{8}} &
	\bullet^{w_1} \ar@/^/[d]^{d_2} \ar@(ul,ur)[]^{f_{1}} \ar@/^/[r]^{f_{2}} \ar@/^/[l]_{f_{5}} &
	\bullet^{w_2} \ar@(ul,ur)[]^{f_{4}} \ar@/^/[l]_{f_{3}} \\
	& & \bullet_{u} \ar@/^/[u]^{d_1} \ar@(dl,ul) \ar@(ur,dr) & 
}
\end{align*}
\end{example}

The strategy for obtaining the result is as follows.   
By \cite{MR1340839}, the graph \cas $C^*(\mathbf{E}_*)$ and $C^*(\mathbf{E}_{**})$ are isomorphic.  We first show in Proposition \ref{prop:csdouble-new} that $C^*(\mathbf{E}_*)$ and $C^*(\mathbf{E}_{**})$ are still isomorphic if we do not enforce the summation relation at $v_1$ and $w_1$ respectively, by appealing to general classification results. In fact, we need to establish (Lemma \ref{lem:csmurray-new}) that they are isomorphic in a way sending prescribed 
elements of the nonstable $K$-theory to other prescribed 
elements. Using this, we prove in Theorem \ref{t:cuntz-splice-1} by use of Kirchberg's  Embedding Theorem that 
Cuntz splicing once and twice yields isomorphic graph \cas. Finally, we  establish in Proposition \ref{prop:cuntzsplicetwice} that the graph obtained by Cuntz splicing twice is move equivalent to the original, and the desired conclusion follows.

\begin{proposition} \label{prop:csdouble-new}
The relative graph \cas (in the sense of Muhly-Tomforde \cite{MR2054981}) $C^*(\mathbf{E}_{*}, \{v_2\})$ and $C^*(\mathbf{E}_{**}, \{ w_2,w_3,w_4 \})$ are isomorphic. 
\end{proposition}

\begin{proof}
Following \cite[Definition~3.6]{MR2054981} we define a graph 
\begin{align*}
(\mathbf{E}_*)_{\{v_2\}} \ = \ \ \ \ \xymatrix{
  \bullet^{v_1} \ar[d]_{e_{1}'} \ar@(ul,ur)[]^{e_{1}} \ar@/^/[r]^{e_{2}} & \bullet^{v_2} \ar@(ul,ur)[]^{e_{4}} \ar@/^/[l]_{e_{3}} \ar@/^/[dl]^{e_{4}'} \\
  \bullet_{v_1'} & 
}
\end{align*} 
Then by \cite[Theorem~3.7]{MR2054981} we have that $C^*(\mathbf{E}_{*}, \{v_2\}) \cong C^*((\mathbf{E}_*)_{\{v_2\}})$.
Similarly we define a graph 
\begin{align*}
(\mathbf{E}_{**})_{\{w_2,w_3,w_4\}} \ = \ \ \ \ \xymatrix{
	\bullet^{w_4} \ar@(ul,ur)[]^{f_{10}}  \ar@/^/[r]^{ f_{9} } &
	\bullet^{w_3} \ar@/_/[dr]_{f_6'} \ar@(ul,ur)[]^{f_{7}} \ar@/^/[r]^{ f_{6} }  \ar@/^/[l]_{f_{8}} &
	\bullet^{w_1} \ar[d]_{f_1'} \ar@(ul,ur)[]^{f_{1}} \ar@/^/[r]^{f_{2}} \ar@/^/[l]_{f_{5}} &
	\bullet^{w_2} \ar@/^/[dl]^{f_3'} \ar@(ul,ur)[]^{f_{4}} \ar@/^/[l]_{f_{3}} \\
	& & \bullet_{w_1'} & 
}
\end{align*}
Using \cite[Theorem~3.7]{MR2054981} again, we have that $C^*(\mathbf{E}_{**}, \{ w_2,w_3,w_4 \})$ is isomorphic to $C^*((\mathbf{E}_{**})_{\{w_2,w_3,w_4\}})$.

Both the graphs $(\mathbf{E}_*)_{\{v_2\}}$ and $(\mathbf{E}_{**})_{\{w_2,w_3,w_4\}}$ satisfy Condition~(K). 
Using the well developed theory of ideal structure and $K$-theory for graph \cas, we see that both have exactly one nontrivial ideal, that this ideal is the compact operators, and that their six-term exact sequences are 
\begin{align*}
\xymatrix{
	\Z \langle v_1' \rangle \ar[r] & \Z \ar[r] & 0 \ar[d] \\
	0 \ar[u] & \ar[l] 0 & \ar[l] 0
}
\ \ \ \ &  \ \ \ \
\xymatrix{
	\Z \langle w_1' \rangle \ar[r] & \Z \ar[r] & 0 \ar[d] \\
	0 \ar[u] & \ar[l] 0 & \ar[l] 0
}
\end{align*}

Furthermore, in $K_0(C^*((\mathbf{E}_*)_{\{v_2\}}))$ we have 
\begin{align*}
	[p_{v_1}] &= -[p_{v_1'}] = [p_{v_2}],
\end{align*}
and in $K_0(C^*((\mathbf{E}_{**})_{\{w_2,w_3,w_4\}}))$ we have 
\begin{align*}
	[p_{w_1}] &= -[p_{w_1'}] = [p_{w_2}], \\
	[p_{w_3}] &= 0 = [p_{w_4}].
\end{align*}
Therefore the class of the unit is $-[p_{v_1'}]$ and $-[p_{w_1'}]$, respectively. 
It now follows from \cite[Theorem~2]{MR1396721} (see also \cite[Corollary~4.20]{arXiv:1301.7695v1}) that $C^*((\mathbf{E}_*)_{\{v_2\}}) \cong C^*((\mathbf{E}_{**})_{\{w_2,w_3,w_4\}})$ and hence that $C^*(\mathbf{E}_{*}, \{v_2\}) \cong C^*(\mathbf{E}_{**}, \{ w_2,w_3,w_4 \})$.
\end{proof}

We also need a technical result about the projections in $\mathcal{E} = C^*(\mathbf{E}_{*}, \{v_2\})$.

\begin{lemma} \label{lem:csmurray-new}
Let $\mathcal{E} = C^*(\mathbf{E}_{*}, \{v_2\})$ and choose an isomorphism between $\mathcal{E}$ and $C^*(\mathbf{E}_{**}, \{ w_2,w_3,w_4 \})$ according to the previous proposition. 
Let $p_{v_1}$, $p_{v_2}$, $s_{e_1}$, $s_{e_2}$, $s_{e_3}$, $s_{e_4}$ be the canonical generators of $C^*(\mathbf{E}_{*}, \{v_2\})=\mathcal{E}$ and let $p_{w_1}$, $p_{w_2}$, $p_{w_3}$, $p_{w_4}$, $s_{f_1}$, $s_{f_2}, \ldots, s_{f_{10}}$ denote the image of the canonical generators of $C^*(\mathbf{E}_{**}, \{ w_2,w_3,w_4 \})$ in $\mathcal{E}$ under the chosen isomorphism. 
Then 
\begin{align*}
	s_{e_1} s_{e_1}^* + s_{e_2} s_{e_2}^* &\sim s_{f_1} s_{f_1}^* + s_{f_2} s_{f_2}^* +s_{f_5} s_{f_5}^*, \\
	p_{v_1} - \left( s_{e_1} s_{e_1}^* + s_{e_2} s_{e_2}^* \right) &\sim p_{w_1} - \left( s_{f_1} s_{f_1}^* + s_{f_2} s_{f_2}^* +s_{f_5} s_{f_5}^* \right), \\
	1_\mathcal{E}-p_{v_1} = p_{v_2} &\sim p_{w_2}+p_{w_3}+p_{w_4}=1_\mathcal{E}-p_{w_1}
\end{align*}
in $\mathcal{E}$, where $\sim$ denotes Murray-von~Neumann equivalence. Thus there exists a unitary $z_0$ in $\mathcal{E}$ such that
\begin{align*}
	z_0\left(s_{e_1} s_{e_1}^* + s_{e_2} s_{e_2}^*\right)z_0^* &= s_{f_1} s_{f_1}^* + s_{f_2} s_{f_2}^* +s_{f_5} s_{f_5}^*, \\
	z_0\left(p_{v_1} - \left( s_{e_1} s_{e_1}^* + s_{e_2} s_{e_2}^* \right)\right)z_0^* &= p_{w_1} - \left( s_{f_1} s_{f_1}^* + s_{f_2} s_{f_2}^* +s_{f_5} s_{f_5}^* \right), \\
	z_0p_{v_1}z_0^* &= p_{w_1} \\	
	 z_0p_{v_2}z_0^* &= p_{w_2}+p_{w_3}+p_{w_4}.
\end{align*}
\end{lemma}

\begin{proof}
By \cite[Corollary~7.2]{MR2310414}, row-finite graph \cas have stable weak cancellation, so by \cite[Theorem~3.7]{MR2054981}, $\mathcal{E}$ has stable weak cancellation. 
Hence any two projections in $\mathcal{E}$ are Murray-von~Neumann equivalent if they generate the same ideal and have the same $K$-theory class. 

As in the proof of Proposition \ref{prop:csdouble-new}, we will use \cite[Theorem~3.7]{MR2054981} to realize our relative graph \cas as graph \cas of the graphs $(\mathbf{E}_*)_{\{v_2\}}$ and $(\mathbf{E}_{**})_{\{w_2,w_3,w_4\}}$. 
Denote the image of the vertex projections of $C^*((\mathbf{E}_*)_{\{v_2\}})$ inside $\mathcal{E}$ under this isomorphism by $q_{v_1}, q_{v_2}, q_{v_1'}$ and denote the image of the vertex projections of $(\mathbf{E}_{**})_{\{w_2,w_3,w_4\}}$ inside $\mathcal{E}$ under the isomorphisms $(\mathbf{E}_{**})_{\{w_2,w_3,w_4\}}\cong C^*(\mathbf{E}_{**}, \{ w_2,w_3,w_4 \}) \cong\mathcal{E}$ by $q_{w_1}, q_{w_2}, q_{w_3}, q_{w_4}, q_{v_1'}$. 
Using the description of the isomorphism in \cite[Theorem~3.7]{MR2054981}, we see that we need to show that $q_{v_1} \sim q_{w_1}$, $q_{v_1'} \sim q_{w_1'}$ and $q_{v_2}\sim q_{w_2}+q_{w_3}+q_{w_4}$.

Since $(\mathbf{E}_*)_{\{v_2\}}^0$ satisfies Condition~(K) and the smallest hereditary and saturated subset containing $v_1$ is all of $(\mathbf{E}_*)_{\{v_2\}}^0$ we have that $q_{v_1}$ is a full projection (\cite[Theorem~4.4]{MR1988256}).
Similarly $q_{w_1}$, $q_{v_2}$ and $q_{w_2}+q_{w_3}+q_{w_4}$ are full. 
In $K_0(\mathcal{E})$ we have, using our calculations from the proof of Proposition \ref{prop:csdouble-new}, that 
\begin{align*}
	[q_{v_1}] &= [1] = [q_{w_1}], \\
	[q_{v_2}] &= [1] = [q_{w_2}]=[q_{w_2}]+[q_{w_3}]+[q_{w_4}].
\end{align*}
So by stable weak cancellation $q_{v_1} \sim q_{w_1}$ and $q_{v_2}\sim q_{w_2}+q_{w_3}+q_{w_4}$.

Both $q_{v_1'}$ and $q_{w_1'}$ generate the only nontrivial ideal $\mathfrak{I}$ of $\mathcal{E}$ (\cite[Theorem~4.4]{MR1988256}). 
Since that ideal is isomorphic to the compact operators and both $[q_{v_1'}]$ and $[q_{w_1'}]$ are positive generators of  $K_0(\mathfrak{I})\cong K_0(\mathbb{K})\cong \Z$, they must both represent the same class in $K_0(\mathfrak{I})$, and thus also in $K_0(\mathcal{E})$. 
Therefore $q_{v_1'} \sim q_{w_1'}$.

Let $u$, $v$ and $w$ be partial isometries realizing the Murray-von~Neumann equivalences. Then $z_0=u+v+w$ is a unitary that satisfies the required conditions. 
\end{proof}

\begin{theorem}\label{t:cuntz-splice-1}
Let $E$ be a graph and let $u$ be a vertex of $E$.  Then $C^{*}(E_{u,-}) \cong C^{*}(E_{u,--})$.
\end{theorem}

\begin{proof}
As above, we let $\mathcal{E}$ denote the \ca $C^*(\mathbf{E}_{*}, \{v_2\})$, and we choose an isomorphism between $\mathcal{E}$ and $C^*(\mathbf{E}_{**}, \{ w_2,w_3,w_4 \})$, which exists according to Proposition~\ref{prop:csdouble-new}.

Since $C^*(E_{u,-})$ and $\mathcal{E}$ are separable, nuclear \cas, by the Kirchberg Embedding Theorem \cite{MR1780426}, there exist an injective \starhomo
\[
	C^*(E_{u,-}) \oplus \mathcal{E} \hookrightarrow \mathcal{O}_2.
\]
We will suppress this embedding in our notation. 

In $\mathcal{O}_2$, we denote the vertex projections and the partial isometries coming from $C^*(E_{u,-})$ by $p_v, v\in E_{u,-}^0$ and $s_e,e\in E_{u,-}^1$, respectively, and we denote the vertex projections and the partial isometries coming from $\mathcal{E}=C^*(\mathbf{E}_*, \{v_2\})$ by $p_1,p_2$ and $s_1, s_2, s_3, s_4$, respectively.
Since we are dealing with an embedding, it follows from Szyma{\'n}ski's Generalized Cuntz-Krieger Uniqueness Theorem (\cite[Theorem~1.2]{MR1914564}) that for any vertex-simple cycle $\alpha_1 \alpha_2 \cdots \alpha_n$ in $E_{u,-}$ without any exit, we have that the spectrum of $s_{\alpha_1} s_{\alpha_2} \cdots s_{\alpha_n}$ contains the entire unit circle.

We will define a new Cuntz-Krieger $E_{u,-}$-family. 
We let
\begin{align*}
q_v&=p_v&&\text{for each }v\in E^0, \\
q_{v_1}&=p_1, \\
q_{v_2}&=p_2.
\end{align*}
Since any two nonzero projections in $\mathcal{O}_2$ are Murray-von~Neumann equivalent, we can choose partial isometries $x_1, x_2 \in \mathcal{O}_2$ such that 
\begin{align*}
	x_1 x_1^* &= s_{d_1} s_{d_1}^*				& x_1^* x_1 &= p_1 \\
	x_2 x_2^* &= p_1 - (s_1 s_1^* + s_2 s_2^*)	& x_2^* x_2 &= p_u.
\end{align*} 
We let 
\begin{align*}
t_e&=s_e&&\text{for each }e\in E^1, \\
t_{e_i}&=s_i&&\text{for each }i=1,2,3,4, \\
t_{d_1}&=x_1, \\
t_{d_2}&=x_2. 
\end{align*}

By construction $\setof{ q_v }{ v \in E_{u,-}^0 }$ is a set of orthogonal projections, and $\setof{ t_e }{ e \in E_{u,-}^1 }$ a set of partial isometries.
Furthermore, by choice of $\setof{ t_e }{ e \neq d_1,d_2 }$ the relations are clearly satisfied at all vertices other than $v_1$ and $u$. 
The choice of $x_1, x_2$ ensures that the relations hold at $u$ and $v_1$ as well. 
Hence $\{q_v, t_e\}$ does indeed form a Cuntz-Krieger $E_{u,-}$-family. 
Denote this family by $\mathcal{S}$.

Using the universal property of graph \cas, we get a $*$-homomorphism from $C^*(E_{u,-})$ onto $C^*(\mathcal{S}) \subseteq \mathcal{O}_2$.
Let $\alpha_1 \alpha_2 \cdots \alpha_n$ be a vertex-simple cycle in $E_{u,-}$ without any exit. 
Since $u$ is where the Cuntz splice is glued on, no vertex-simple cycle without any exit uses edges connected to $u, v_1$ or $v_2$. 
Hence $t_{\alpha_1} t_{\alpha_2} \cdots t_{\alpha_n} = s_{\alpha_1} s_{\alpha_2} \cdots s_{\alpha_n}$ and so its spectrum contains the entire unit circle. 
It now follows from \cite[Theorem~1.2]{MR1914564} that the \starhomo from $C^*(E_{u,-})$ to $C^*(\mathcal{S})$ is in fact a \stariso. 

Let $\A_0$ be the \ca generated by $\setof{ p_v }{ v \in E^0 }$, and let \A be the subalgebra of $\mathcal{O}_2$ generated by $\setof{ p_v }{ v \in E^0 }$ and $\mathcal{E}$.  Note that $\A=\A_0\oplus\mathcal{E}$.

Let us denote by $\setof{ r_{w_i}, y_{f_j} }{ i=1,2,3,4, j = 1,2,\ldots, 10 }$ the image of the canonical generators of $C^*(\mathbf{E}_{**}, \{ w_2,w_3,w_4 \})$ in $\mathcal{O}_2$ under the chosen isomorphism between $C^*(\mathbf{E}_{**}, \{ w_2,w_3,w_4 \})$ and $\mathcal{E}$ composed with the embedding into $\mathcal{O}_2$. 

By Lemma \ref{lem:csmurray-new}, certain projections in $\mathcal{E}$ are Murray-von~Neumann equivalent, so choose a unitary $z_0 \in \mathcal{E}$ according to this lemma, and set $z=z_0+\sum_{v\in E^0}p_v \in \mathcal{M}(\A)$.  Clearly $z$ is a unitary in $\mathcal{M}( \A )$.  Since the approximate identity of $\A$ given by
\[
\left\{ \sum_{ k = 1 }^n p_{v_k} + 1_\mathcal{E} \right\}_{n \in \N },
\] 
where $\setof{ p_v }{ v \in E^0 } = \{ p_{v_1} , p_{v_2} , \dots \}$, is an approximate identity of $C^*( \mathcal{S} )$, we have a canonical unital $^*$-homomorphism from $\mathcal{M} ( \A )$ to $\mathcal{M} ( C^* (\mathcal{S} ) )$ which, when restricted to $\A$, gives the embedding of $\A$ into $C^*( \mathcal{S})$.  So we can consider $z$ as a unitary in $\mathcal{M} ( C^* ( \mathcal{S} ) )$.  Hence, for all $x \in C^*(\mathcal{S})$, we have that $z x$ and $xz$ are elements of $C^* ( \mathcal{S} )$.  By construction of $z$, we have that 
\begin{align*}
	z q_v  &= q_vz=q_v, \text{ for all } v \in E^0, \\
	z t_e  &= t_ez=t_e, \text{ for all } e \in E^1, \\
	z \left( t_{e_1} t_{e_1}^* + t_{e_2} t_{e_2}^* \right) z^* &= y_{f_1} y_{f_1}^* + y_{f_2} y_{f_2}^* +y_{f_5} y_{f_5}^*, \\
	z \left( q_{v_1} - \left( t_{e_1} t_{e_1}^* + t_{e_2} t_{e_2}^* \right) \right) z^* &= r_{w_1} - \left( y_{f_1} y_{f_1}^* + y_{f_2} y_{f_2}^* +y_{f_5} y_{f_5}^* \right), \\
z q_{v_1} z^* &= r_{w_1}, \\ 
z q_{v_2} z^* &= r_{w_2}+r_{w_3}+r_{w_4}.
\end{align*}

We will now define a Cuntz-Krieger $E_{u,--}$-family in $\mathcal{O}_2$.
We let
\begin{align*}
P_v &= q_v = p_v &&\text{for each }v\in E^0, \\
P_{w_i}&=r_{w_i} &&\text{for each }i=1,2,3,4.
\end{align*}
Moreover, we let
\begin{align*}
S_e&=t_e=s_e &&\text{for each }e\in E^1, \\
S_{f_i}&=y_{f_i} &&\text{for each }i=1,2,\ldots,10, \\
S_{d_1}&=zt_{d_1}z^*=zx_1z^*,\\
S_{d_2}&=zt_{d_2}z^*=zx_2z^*.
\end{align*}
Denote this family by $\mathcal{T}$.

By construction $\setof{ P_v }{ v \in E_{u,--}^0 }$ is a set of orthogonal projections, and $\setof{ S_e }{ e \in E_{u,--}^1}$ a set of partial isometries satisfying
\begin{align*}
S_e^*S_e&=s_e^*s_e=p_{r(e)}, 
& S_eS_e^*&=s_es_e^*, \\
S_{f_i}^*S_{f_i}&=y_{f_i}^*y_{f_i}=r_{f_i}, & 
S_{f_i}S_{f_i}^*&=y_{f_i}y_{f_i}^*, \\ 
S_{d_1}^*S_{d_1}&=r_{w_1}, & S_{d_1}S_{d_1}^*&=s_{d_1}s_{d_1}^*, \\ 
S_{d_2}^*S_{d_2}&=p_u, & S_{d_2}S_{d_2}^*&=r_{w_1}-\left(y_{f_1}y_{f_1}^*+y_{f_2}y_{f_2}^*+y_{f_5}y_{f_5}^*\right), 
\end{align*}
for all $e\in E^1$ and $i=1,2,\ldots,10$. From this, it is clear that $\mathcal{T}$ will satisfy the Cuntz-Krieger relations at all vertices in $E^0$.
Similarly, we see that since $\setof{ r_{w_i}, y_{f_j} }{ i=1,2,3,4, j = 1,2,\ldots, 10 }$ is a Cuntz-Krieger $(\mathbf{E}_{**}, \{ w_2,w_3,w_4 \})$-family, $\mathcal{T}$ will satisfy the relations at the vertices $w_2, w_3, w_4$. 
It only remains to check the summation relation at $w_1$, for that we compute
\begin{align*}
	\smash{\sum_{s_{E_{u,--}}(e) = w_1} S_e S_e^*}	&= S_{f_1} S_{f_1}^* + S_{f_2} S_{f_2}^* + S_{f_5} S_{f_5}^* + S_{d_2} S_{d_2}^* \\
											&= y_{f_1} y_{f_1}^* + y_{f_2} y_{f_2}^* +y_{f_5} y_{f_5}^* + r_{w_1}-\left(y_{f_1}y_{f_1}^*+y_{f_2}y_{f_2}^*+y_{f_5}y_{f_5}^*\right) \\
											&= r_{w_1} = P_{w_1}.
\end{align*}
Hence $\mathcal{T}$ is a Cuntz-Krieger $E_{u,--}$-family. 

The universal property of $C^*(E_{u,--})$ provides a surjective \starhomo from $C^*(E_{u,--})$ to $C^*(\mathcal{T}) \subseteq \mathcal{O}_2$. 
Let $\alpha_1 \alpha_2 \cdots \alpha_n$ be a vertex-simple cycle in $E_{u,--}$ without any exit. 
We see that all the edges $\alpha_i$ must be in $E^1$, and hence we have
\begin{align*}
	S_{\alpha_1} S_{\alpha_2} \cdots S_{\alpha_n} = t_{\alpha_1}t_{\alpha_2} \cdots t_{\alpha_n} = s_{\alpha_1} s_{\alpha_2} \cdots s_{\alpha_n}  
\end{align*} 
and so its spectrum contains the entire unit circle. 
It now follows from \cite[Theorem~1.2]{MR1914564} that $C^*(E_{u,--})$ is isomorphic to $C^*(\mathcal{T})$.

Recall that $z \in \mathcal{M}( C^* ( \mathcal{S} ) )$.  Therefore, $\mathcal{T} \subseteq C^*(\mathcal{S})$ since $\A \subseteq C^*(\mathcal{S})$ and since $r_{w_i}, y_{f_j}\in \mathcal{E} \subseteq C^*(\mathcal{S})$, for $i=1,2,3,4$, $j = 1,2,\ldots, 10$.  So $C^*(\mathcal{T}) \subseteq C^*(\mathcal{S})$.  

Since the approximate identity of $\A$ given by
\[
\left\{ \sum_{ k = 1 }^n p_{v_k} + 1_\mathcal{E} \right\}_{n \in \N },
\] 
where $\setof{ p_v }{ v \in E^0 } = \{ p_{v_1} , p_{v_2} , \dots \}$, is an approximate identity of $C^*( \mathcal{T} )$, we get that for all $x \in C^*(\mathcal{T})$, $z x z^*$ and $z^* x z$ are elements of $C^* ( \mathcal{T} )$.  But since \A is also contained in $C^*(\mathcal{T})$ and $\mathcal{E} \subseteq C^*(\mathcal{T})$, we have that $\mathcal{S} \subseteq C^*(\mathcal{T})$, and hence $C^*(\mathcal{S}) \subseteq C^*(\mathcal{T})$. 
Therefore 
\[
	C^*(E_{u,-}) \cong C^*(\mathcal{S}) = C^*(\mathcal{T}) \cong C^*(E_{u,--}).\qedhere
\]
\end{proof}

The next two results will show that $E \Meq E_{u, -- }$ for a row-finite graph $E$ and a vertex $u \in E^0$ which supports two distinct return paths.   This will be enough to prove our main result since by \cite[Lemma~5.1]{arXiv:1510.06757v2}, there exists a row-finite graph $F$ and a vertex $v$ supporting two distinct return paths such that $C^* ( E_{u, - } ) \otimes \K \cong C^* ( F_{v, - } ) \otimes \K$ and $C^* ( E ) \otimes \K \cong C^* ( F ) \otimes \K$. 

\begin{proposition}\label{prop:cuntzSpliceSetup}Let $E$ be a row-finite graph and let $u$ be a vertex which supports two distinct return paths.  Then there exists a row-finite graph $F$ and a vertex $v \in F^0$ which supports two distinct loops such that $E \Meq F$ and $E_{u, -- } \Meq F_{v, -- }$.
\end{proposition}

\begin{proof}
Throughout the proof, we will freely use the following fact:  Let $G$ be a graph and let $w$ be a vertex and let $w' \neq w$ be a regular vertex that does not support a loop.  Let $G'$ be the resulting graph after collapsing $w'$.  Then $G \Meq G'$ and $G_{w, -- } \Meq G'_{w, -- }$.  

Suppose $u \in E^0$ supports two loops.  Then set $E = F$ and $v = u$.  Suppose $u$ does not support two loops.  Then there exists a return path $\mu = e_1 e_2 \cdots e_n$ with $n \geq  2$.  Starting at $r(e_1)$, if $r( e_1 )$ does not support a loop, we collapse the vertex $r( e_1)$.  Doing this will result in reducing the length on $\mu$.  Note that we may have also added a loop at $u$.  Continuing this procedure, we have obtained a graph $E'$ with $u$ in $(E')^0$ such that $E \Meq E'$,  $E_{u, -- } \Meq E'_{u, -- }$, and either $u$ supports two loops or $u$ supports a return path $\nu = f_1 f_2 \dots f_m$ with $m \geq 2$, with $r( f_1 )$ supporting a loop.

If $u$ supports two loops, set $F = E'$ and $v = u$.  Suppose $u$ supports a return path $\nu = f_1 f_2 \dots f_m$ with $m \geq 2$, with $r( f_1 )$ supporting a loop.  Then by Proposition~\ref{prop:columnAdd}, we add the $r( f_1 )$'th column to the $u$'th column twice, to get a graph $F$ with $u \in F^0$ supporting two loops such that $F \Meq E'$.  Note that we may perform the same matrix operations to $\Bsf_{E'_{u, -- } }$ and get that $E'_{u, -- } \Meq F_{u, -- }$.  Set $v = u$.  

We have just obtained the desired graph $F$ and the desired vertex $v \in F^0$ since $E \Meq E' \Meq F$ and $E_{u, -- } \Meq E'_{u , -- } \Meq F_{v, -- }$.
\end{proof}

We now show that performing the Cuntz splice twice is a legal move for a row-finite graph.

\begin{proposition} \label{prop:cuntzsplicetwice}
Let $E$ be a row-finite graph and let $v$ be a vertex that supports at least two distinct return paths. 
Then $E \Meq E_{v,--}$.  
\end{proposition}
\begin{proof}
According to Proposition~\ref{prop:cuntzSpliceSetup}, we can assume that $E$ satisfies the conditions of that proposition --- so we assume that $v$ is a regular vertex that supports at least two loops. 

For a given matrix size $N \in \N \cup \{ \infty \}$ and $i,j\in\{1,2,\ldots,N\}$, we let $E_{(i,j)}$ denote the $N\times N$ matrix that is equal to the identity matrix everywhere except for the $(i,j)$'th entry, that is $1$.
If $B$ is a $N\times N$ matrix, then $E_{(i,j)}B$ is the matrix obtained from $B$ by adding $j$'th row into the $i$'th row, and $BE_{(i,j)}$ is the matrix obtained from $B$ by adding $i$'th column into the $j$'th column. 
Using $E_{(i,j)}^{-1}$ instead will yield subtraction. 
In what follows we will make extensive use of the results from Section \ref{addmoves}, but we do so implicitly since we feel it will only muddle the exposition if we add all the references in. 

Note that $\Bsf_{E_{v,--}}$ can be written as 
$$B_1=
\begin{pmatrix}
\begin{pmatrix}
0 & 1 & 0 & 0 \\
1 & 0 & 1 & 0 \\
0 & 1 & 0 & 1 \\
0 & 0 & 1 & 0 
\end{pmatrix}  &
\begin{pmatrix} 
0 & 0 & \cdots  \\
0 & 0 & \cdots  \\
1 & 0 & \cdots  \\
0 & 0  & \cdots  
\end{pmatrix} \\
\begin{pmatrix}
0 & 0 & 1 & 0 \\
0 & 0 & 0 & 0\\ 
\vdots &\vdots &\vdots &\vdots \\
\end{pmatrix} & 
\Bsf_E
\end{pmatrix}
$$

Now let 
$B_2=E_{(3,2)}B_1$ and 
$B_3=B_2E_{(2,1)}^{-1}$. 
Then $B_1 + I \Meq B_2 + I \Meq B_3 + I$. 
We have that 
$$
B_3 = 
\begin{pmatrix}
\begin{pmatrix}
-1& 1 & 0 & 0 \\
1 & 0 & 1 & 0 \\
0 & 1 & 1 & 1 \\
0 & 0 & 1 & 0
\end{pmatrix} &  \begin{pmatrix} 
0 & 0 & \cdots  \\
0 & 0 &\cdots  \\
1 & 0 & \cdots  \\
0 & 0 & \cdots  
\end{pmatrix} \\
\begin{pmatrix}
0 & 0 & 1 & 0 \\
0 & 0 & 0 & 0 \\
\vdots &\vdots &\vdots &\vdots \\
\end{pmatrix} & 
\Bsf_E
\end{pmatrix}
$$
The $1$st vertex in $\Esf_{B_3 + I}$ does not support a loop, so it can be collapsed yielding
$$
B_4 = 
\begin{pmatrix}
\begin{pmatrix}
1 & 1 & 0 \\
1 & 1 & 1  \\
 0 & 1 & 0
\end{pmatrix} &  \begin{pmatrix} 
0 & 0 & \cdots  \\
1 & 0 & \cdots  \\
0 & 0 & \cdots  
\end{pmatrix} \\
\begin{pmatrix}
 0 & 1 & 0 \\
  0 & 0 & 0 \\
  \vdots &\vdots &\vdots \\
\end{pmatrix} & 
\Bsf_E
\end{pmatrix}
$$

With $B_4 + I \Meq B_3 + I$. 
Now we let 
$B_5=E_{(2,3)}^{-1}B_4$, 
$B_6=E_{(4,1)}B_5$, 
$B_7=E_{(4,3)}^{-1}E_{(4,3)}^{-1}B_6$, 
$B_8=E_{(1,2)}B_7$ and 
$B_9=B_8E_{(2,3)}^{-1}$. 
We then have 
$B_4 + I \Meq B_5 + I \Meq B_6 + I \Meq B_7 + I \Meq B_8 + I\Meq B_9 + I$. 
We have that 
$$
B_9 = 
\begin{pmatrix}
\begin{pmatrix}
2 & 1 & 0 \\
1 & 0 & 1  \\
 0 & 1 & -1
\end{pmatrix} &  \begin{pmatrix} 
1 & 0 & \cdots  \\
1 & 0 & \cdots  \\
0 & 0 & \cdots  
\end{pmatrix} \\
\begin{pmatrix}
 1 & 0 & 0 \\
  0 & 0 & 0 \\
\vdots &\vdots &\vdots \\
\end{pmatrix} & 
\Bsf_E
\end{pmatrix}
$$
In $\Esf_{B_9 + I}$ the $3$rd vertex does not support a loop, so it can be collapsed to yield
$$
B_{10} = 
\begin{pmatrix}
\begin{pmatrix}
2 & 1  \\
1 & 1  
\end{pmatrix} &  \begin{pmatrix} 
1 & 0 & \cdots  \\
1 & 0 & \cdots  \\
\end{pmatrix} \\
\begin{pmatrix}
 1 & 0  \\
 0 & 0  \\
\vdots &\vdots  \\
\end{pmatrix} & 
\Bsf_E
\end{pmatrix}
$$
with $B_9 + I \Meq B_{10} + I$. 

Now we look at the graph $E$ again, and and let $\Bsf_E = (b_{ij})$. 
Since the vertex $v$ (number $1$) has at least two loops, we have $b_{11}\geq 1$. 
Now we can insplit by partitioning $r^{-1}(v)$ into two sets, one with a single set consisting of a loop based at $v$, and the other the rest. 
In the resulting graph, $v$ is split into two vertices $v^1$ and $v^2$, and let $E'$ denote the rest of the graph. The vertex $v^1$ has the same edges in and out of $E'$ as $v$ had, but it has only $b_{11}$ loops. There is one edge from $v^1$ to $v^2$ and $v^2$ has one loop and there are $b_{11}$ edges from $v^2$ to $v^1$ as well as all the same edges going from $v^2$ into $E'$ as originally from $v$. 
Use the inverse collapse move to add a new vertex $u$ to the middle of the edge from $v^1$ to $v^2$ and call the resulting graph $F$. 
Label the vertices such that  $v^2$, $u$ and $v^1$ are the $1$st, $2$nd and $3$rd vertices, then $\Bsf_F$ is:
$$\Bsf_F =
\begin{pmatrix}
\begin{pmatrix}
0 & 0  \\
1 & -1  
\end{pmatrix} &  \begin{pmatrix} 
b_{11} & b_{12} & \cdots  \\
0 & 0 & \cdots  \\
\end{pmatrix} \\
\begin{pmatrix}
 0 & 1  \\
 0 & 0  \\
\vdots &\vdots  \\
\end{pmatrix} & 
\widetilde{B}
\end{pmatrix}
$$
where $\widetilde{B}$ is $\Bsf_E$ except for on the $(1,1)$'th entry, which is $b_{11}-1$.   
Note that $b_{11}-1\geq 0$, so that there is still a loop based at the $3$rd vertex.  Also, note that since $E$ is a row-finite graph, the $b_{1k}$'s are eventually zero.
This is important since it allows us to do the following matrix manipulations. 
Let $C_2=\Bsf_F E_{(1,2)}E_{(1,2)}$, 
$C_3=E_{(1,2)}C_2$, 
$C_4=E_{(1,3)}^{-1}C_3$,
$C_5=C_4E_{(2,3)}$ and
$C_6=C_5E_{(1,2)}$. 
We have that
$C_1 + I \Meq C_2 + I \Meq C_3 + I \Meq C_4 + I \Meq C_5 + I \Meq C_6 + I$. 
The matrix $C_6$ 
$$
C_6 = 
\begin{pmatrix}
\begin{pmatrix}
1 & 1  \\
1 & 2  
\end{pmatrix} &  \begin{pmatrix} 
1 & 0 & \cdots  \\
1 & 0 & \cdots  \\
\end{pmatrix} \\
\begin{pmatrix}
 0 & 1  \\
 0 & 0  \\
\vdots &\vdots  \\
\end{pmatrix} & 
\Bsf_E
\end{pmatrix}.
$$
Therefore, $C_6$ is in fact equivalent to $B_{10}$ upon relabeling of the first two vertices, thus it follows, that $E\Meq E_{v,--}$. 
\end{proof}

Thus we have the following fundamental result.

\begin{theorem}\label{theorem:cuntzspliceinvariant}
Let $E$ be a graph and let $v$ be a vertex that supports at least two distinct return paths. 
Then $C^*(E)\otimes\K\cong C^*(E_{v,-})\otimes\K$.
\end{theorem}

\begin{proof}
By \cite[Lemma~5.1]{arXiv:1510.06757v2}, we may assume that $E$ is a graph with no singular vertices, in particular, $E$ is a row-finite graph.   By Theorem~\ref{t:cuntz-splice-1}, $C^{*} ( E_{v, -} ) \cong C^{*} ( E_{v,- -} )$ and hence, $C^{*} ( E_{v, -} ) \otimes \K \cong C^{*} ( E_{v,- -} ) \otimes \K$.  By Proposition~\ref{prop:cuntzsplicetwice}, $C^{*} (E) \otimes \K \cong C^{*} ( E_{v, - - } ) \otimes \K$.  Thus, $C^*(E)\otimes\K\cong C^*(E_{v,-})\otimes\K$.
\end{proof}


\section*{Acknowledgements}

This work was partially supported by the Danish National Research Foundation through the Centre for Symmetry and Deformation (DNRF92), by VILLUM FONDEN through the network for Experimental Mathematics in Number Theory, Operator Algebras, and Topology, by a grant from the Simons Foundation (\# 279369 to Efren Ruiz), and by the Danish Council for Independent Research | Natural Sciences. 

Part of this work was initiated while all four authors were
attending the research program \emph{Classification of operator algebras: complexity,
rigidity, and dynamics} at the Mittag-Leffler Institute, January--April 2016.  The authors would also like to thank Aidan Sims for many fruitful discussions.


\providecommand{\bysame}{\leavevmode\hbox to3em{\hrulefill}\thinspace}
\providecommand{\MR}{\relax\ifhmode\unskip\space\fi MR }
\providecommand{\MRhref}[2]{%
  \href{http://www.ams.org/mathscinet-getitem?mr=#1}{#2}
}
\providecommand{\href}[2]{#2}

\end{document}